\documentclass[12pt]{amsart}
\usepackage{amsfonts}
\usepackage{amssymb}
\usepackage[letterpaper, left=2.5cm, right=2.5cm, top=2.5cm,
bottom=2.5cm,dvips]{geometry}
\usepackage{verbatim}
\usepackage{graphicx}
\newtheorem{theorem}{Theorem}
\theoremstyle{plain}

\newtheorem{claim}{Claim}

\newtheorem{conjecture}{Conjecture}
\newtheorem{corollary}{Corollary}

\newtheorem{lemma}{Lemma}

\numberwithin{equation}{section}

\begin{document}
	\title[Extremal Square-free Words]{Extremal Square-free Words}

	\author{Jaros\l aw Grytczuk}
	\address{Faculty of Mathematics and Information Science, Warsaw University
		of Technology, 00-662 Warsaw, Poland}
	\email{j.grytczuk@mini.pw.edu.pl}
	
	\author{Hubert Kordulewski}
	\address{Faculty of Mathematics and Information Science, Warsaw University
		of Technology, 00-662 Warsaw, Poland}
	\email{j.grytczuk@mini.pw.edu.pl}
	
	\author{Artur Niewiadomski}
	\address{Faculty of Mathematics and Information Science, Warsaw University
		of Technology, 00-662 Warsaw, Poland}
	\email{j.grytczuk@mini.pw.edu.pl}

	\begin{abstract}
A word is \emph{square-free} if it does not contain non-empty factors of the form $XX$. In 1906 Thue proved that there exist arbitrarily long square-free words over $3$-letter alphabet. We consider a new type of square-free words. A square-free word is \emph{extremal} if it cannot be extended to a new square-free word by inserting a single letter on arbitrary position. We prove that there exist infinitely many extremal words over $3$-letter alphabet. Some parts of our construction relies on computer verifications. We also pose some related open problems.
	\end{abstract}
	
	\maketitle
	
	\section{Introduction}
	
	A \emph{square} is a non-empty word of the form $XX$. A word is \emph{square-free} if it does not contain a square as a factor. It is easy to check that there are no binary square-free words of length more than $4$. However, there exist ternary square-free words of any length, as proved by Thue in 1906 \cite{Thue}. This result is the starting point of Combinatorics on Words, a wide discipline with lots of exciting problems, deep results, and important applications (see \cite{AlloucheShallit}, \cite{BerstelThue}, \cite{GrytczukDM}, \cite{Lothaire}).
	
	In this paper we propose a new problem of extremal nature in this area. Let $\mathbb A$ be a fixed alphabet and let $W$ be a finite word over $\mathbb A$. An \emph{extension} of $W$ is any word of the form $W'xW''$, where $x\in A$ and $W=W'W''$. A square-free word $w$ is called \emph{extremal} over $\mathbb A$ if there is no square-free extension of $W$. For instance, the word $$H=\mathtt {abcaba cbc abc b abc aba cbc abc}$$is perhaps the shortest extremal word over alphabet $\mathbb A=\{\mathtt{a,b,c}\}$. Our main result asserts that there are infinitely many such words.
	
	\begin{theorem}\label{Theorem Extremal Words}
	There exist arbitrarily long extremal square-free words over $3$-letter alphabet.
	\end{theorem}

The proof is by recursive construction whose validity is based on computer verifications. We will give it in section 2.

A related question concerns the following way of generating square-free words. Given a fixed, ordered alphabet $\mathbb A$, we start with the first letter from $\mathbb A$ and continue by inserting on the rightmost position of the actual word the earliest possible letter so that a new word is square-free. For instance, for alphabet $\mathbb A=\{\mathtt{a,b,c}\}$ this greedy procedure starts with the following sequence of square-free words $$\mathtt {a, ab, aba, abac, abaca, abacab, abacaba, abacabca}.$$ The last word was obtained by inserting $\mathtt c$ on penultimate position of the previous word. We conjecture that the above procedure never stops.

To state it formally, let us define recursively a sequence of  \emph{nonchalant words} $G_i$ over a fixed ordered alphabet $\mathbb A$ by putting $G_1=\mathtt a$, and letting $G_{i+1}=G_i'xG_i''$ be a square-free extension of $G_i$ such that $G_i''$ is the shortest possible suffix of $G_i$ and $x\in \mathbb A$ is the earliest possible letter.

\begin{conjecture}\label{Conjecture Nonchalant}
The sequence of nonchalant words over $3$-letter alphabet is infinite.
\end{conjecture}

In the last section we discuss some numerical experiments concerning this conjecture and other related problems.

\section{Proof of the main result}

We start with a general result on which our construction is based. Consider a finite directed graph $D$ on the set of vertices $V=\{v_1,v_2,\dots,v_n\}$. Suppose that each vertex $v_i$ is labeled with some word $B_i=f(v_i)$ over a fixed alphabet $\mathbb A$. We will refer to these words $B_i$ as \emph{blocks}.

A \emph{walk} in $D$ is any sequence $W=w_1w_2\dots w_t$, with $w_i\in V$, such that $(w_i,w_{i+1})$ is a directed edge of $D$ for every $i=1,2,\dots, t-1$. Every walk $W=w_1w_2\dots w_t$ generates in a natural way a word $f(W)=f(w_1)f(w_2)\dots f(w_t)$ over alphabet $\mathbb A$ by concatenating blocks corresponding to consecutive vertices $w_i$ in $W$. More formally, one may consider $f$ as a homomorphism from monoid $V^*$ to monoid $\mathbb A^*$  defined by a substitution $f(v_i)=B_i$.

A walk is \emph{square-free} if it is a square-free word over alphabet $V$. We say that a digraph $D$ is a \emph{Thue digraph} if for every square free walk $W$, the word $f(W)$ is also square-free (as a word over $\mathbb A$). Let $S(D)$ denote the set of all words over $\mathbb A$ derived as images of any square-free walks in $D$. So, a digraph $D$ is a Thue digraph if $S(D)$ contains only square-free words. The result below gives sufficient conditions for this property.

\begin{theorem}\label{Theorem Thue Digraph}
Let $D$ be a digraph on the vertices $V=\{v_1,v_2,\dots,v_n\}$ labeled with some blocks $B_i=f(v_i)$ over alphabet $\mathbb A$. Then $D$ is a Thue digraph if the following conditions are satisfied:
\item[(1)] For every square-free walk $W=w_1w_2w_3$, the word $f(W)$ is also square-free.
\item[(2)] No block $B_i$ is a factor of another block $B_j$.
\item[(3)] For every pair of distinct blocks $B_i$ and $B_j$, $i\neq j$, and any factorization $B_i=XX'$ and $B_j=YY'$, none of the words $XY'$ nor $X'Y$ can be equal to any block $B_k$, unless $B_k=B_i=X$ or $B_k=B_j=Y$.
\end{theorem}
\begin{proof}
	Suppose for a contradiction that a square $XX$ appears in some word $f(W)$, where $W=w_1w_2\dots w_t$ is a square-free walk in $D$. Assume also that $W$ is a shortest such walk. So, we may write (see Figure \ref{Figure Square}):$$f(W)=PP'f(w_2)\dots f(w_{j})QQ'f(w_{j+2})\dots f(w_{t-1})RR'=PXXR',$$where $f(w_1)=PP'$,$f(w_{j+1})=QQ'$, $f(w_t)=RR',$ and $$X=P'f(w_2)\dots f(w_j)Q=Q'f(w_{j+2})\dots f(w_{t-1})R=X.$$
	
	\begin{figure}[h]
		\center
		\includegraphics[width=1.0\textwidth]{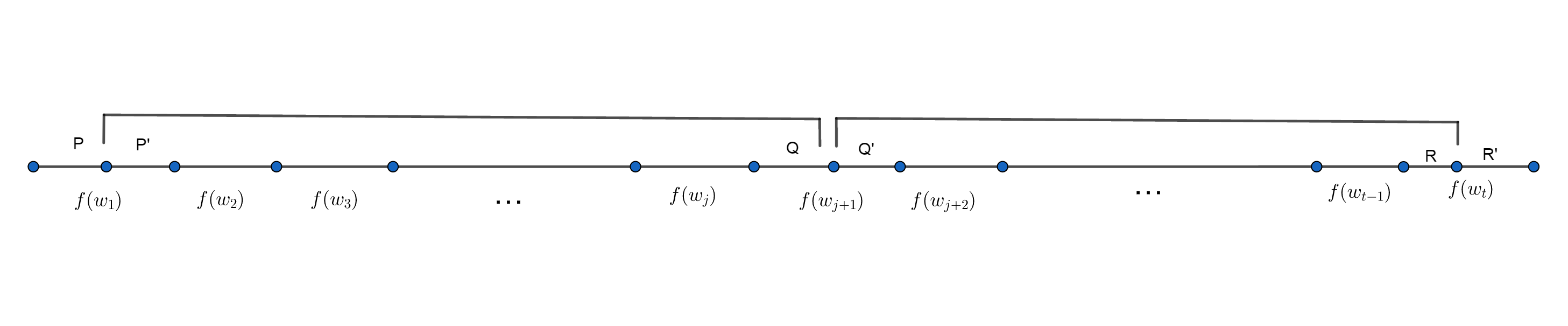}
		\caption{A square in $f(W)$.}\label{Figure Square}
	\end{figure}

	By condition $(1)$, the walk $W$ has at least four vertices, hence, at least one part of a square must contain a full occurrence of some block. With no loss of generality we may assume that this happens in the left part. Also we may assume that this part contains as many blocks as the other part.
	
	Let $q\in \{1,2,\dots,j\}$ be the smallest index such that $w_q\neq w_{j+q}$. There must be at least one such index since otherwise the walk $W$ would contain a square $w_1w_2\dots w_{2j}$, contradicting our assumption. We distinguish two cases.
	
	If $q>1$, then either $f(w_q)$ is a prefix of $f(w_{j+q})$ or the other way around, which contradicts condition $(2)$.
	
	If $q=1$, then $f(w_1)\neq f(w_{j+1})$ and we consider two cases. First suppose that the words $P'$ and $Q'$ have different lengths, and assume that $P'$ is longer. Then we may write $P'=Q'X'$, where $X'$ is a non-empty suffix of the block $f(w_1)$ (see Figure \ref{Figure Blocks}). Now, the block $f(w_{j+2})$ must end before $f(w_2)$ since otherwise $f(w_2)$ would be contained in $f(w_{j+2})$, contradicting condition (2). So, we may write $f(w_{j+2})=X'Y$, where $f(w_2)=YY'$. This contradicts condition (3). If $Q'$ is longer than $P'$ reasoning is similar.
	\begin{figure}[h]
		\center
		\includegraphics[width=1.0\textwidth]{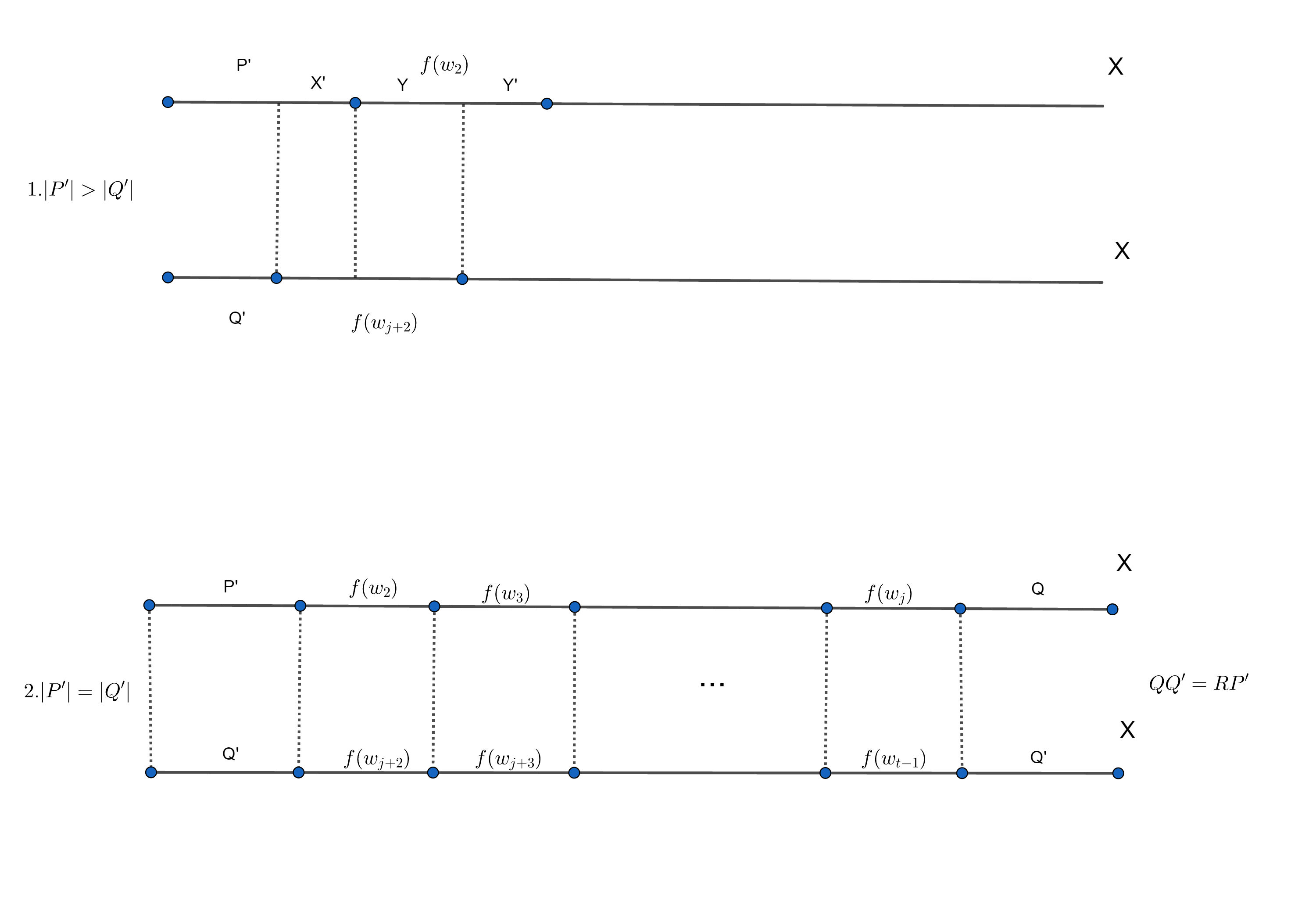}
		\caption{Comparing blocks in the square.}\label{Figure Blocks}
	\end{figure}
	Suppose now that the words $P'$ and $Q'$ have equal length. This implies that all pairs of corresponding inner blocks in the left and right part of the square $XX$ must be equal (otherwise one them would be included in the other one). In consequence we get that also $Q=R$ (see Figure \ref{Figure Blocks}). This means that $f(w_{j+1})=QQ'=RP'$. By condition (3) it follows that either $f(w_{j+1})=f(w_1)$ or $f(w_{j+1})=f(w_t)$. In both cases we get a square in the walk $W$. This completes the proof.
	\end{proof}

Using the above result we may now prove Theorem \ref{Theorem Extremal Words}. First we will construct a Thue digraph on the set of $12$ vertices together with the set of $12$ blocks defined as follows. Consider a square-free word $$N=\mathtt{abacbabcabacbcacbabcabacabcbabcabacbcabcb}.$$This word is \emph{nearly extremal} in the sense that it has only two square-free extensions, namely $\mathtt cN$ and $N\mathtt a$. It is clear that each word obtained from $N$ by permutation of the alphabet and reversal is also nearly extremal. Let us denote the six words corresponding to six permutations of the alphabet as: $$N, N_{\mathtt{ab}}, N_{\mathtt{ac}}, N_{\mathtt{bc}}, N_{\mathtt{abc}}, N_{\mathtt{acb}},$$where indices denotes nontrivial cycles of these permutations. Let us also denote reversals of the above six words by: $$\tilde{N}, \tilde{N}_{\mathtt{ab}}, \tilde{N}_{\mathtt{ac}}, \tilde{N}_{\mathtt{bc}}, \tilde{N}_{\mathtt{abc}}, \tilde{N}_{\mathtt{acb}}.$$

Now we may define a digraph $D_N$ as depicted in Figure \ref{Digraph N}. Its vertices are labeled by the above $12$ nearly extremal words. It can be checked that for each directed edge of $D_N$ the corresponding concatenation of blocks gives a square-free word. Moreover, the following lemma was verified by computer.

\begin{lemma}\label{Lemma N}
	The digraph $D_N$ from Figure \ref{Digraph N} is a Thue digraph.
\end{lemma}

\begin{figure}[h]
	\center
	\includegraphics[width=1.0\textwidth]{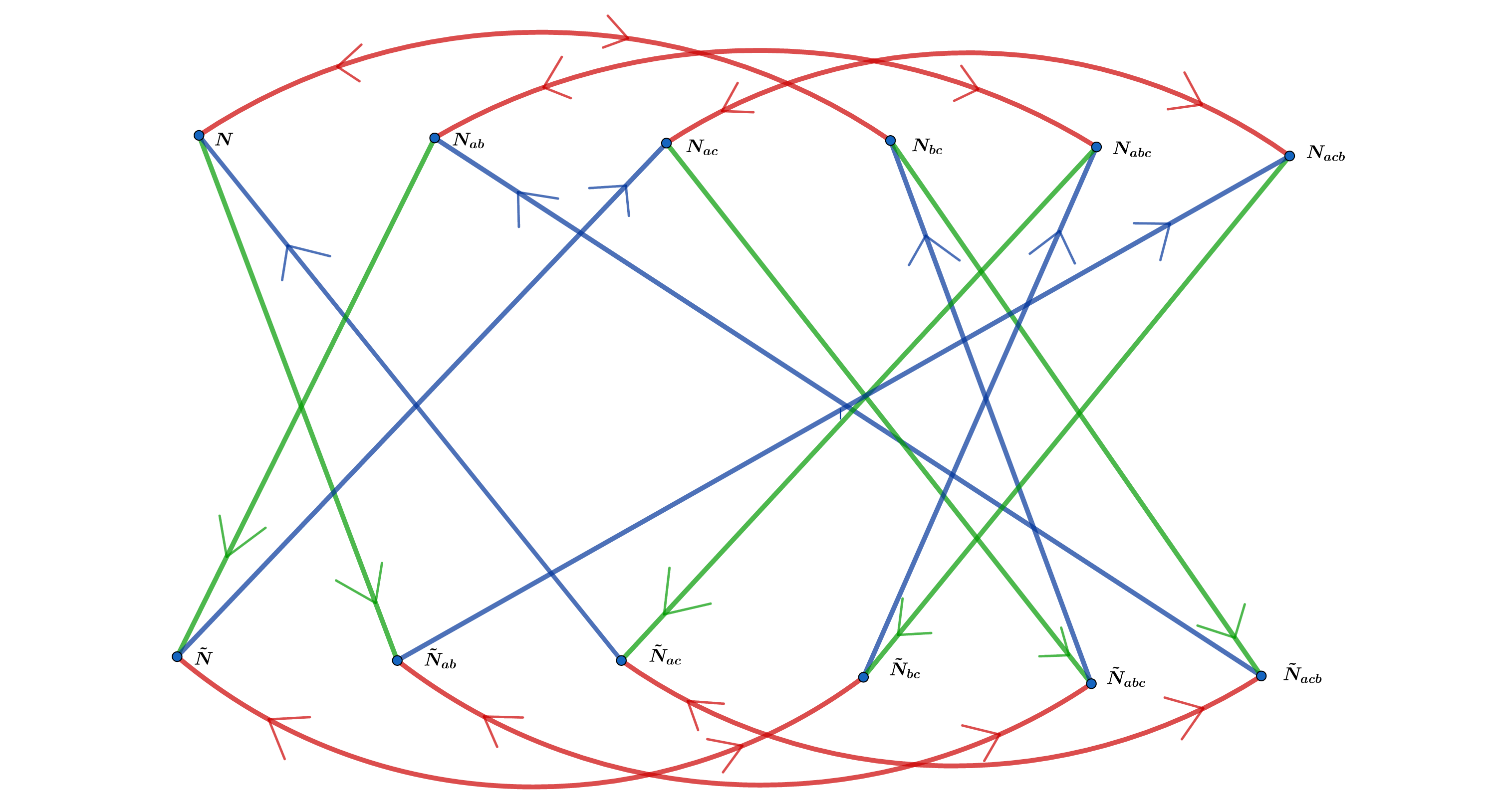}
	\caption{Digraph $D_N$.}\label{Digraph N}
\end{figure}

Recall that the set $S(D_N)$ consists of all words derived as homomorphic images of all square-free walks in $D_N$. By the above lemma, the set of words $S(D_N)$ consists only of square-free words. We are going to prove that this set is infinite. We will need the following general lemmas.

\begin{lemma}\label{Lemma Substitution}
Let $T=a_1a_2\dots a_s$ be a square free word over alphabet $\mathbb A=\{1,2,3\}$, and let $V_1, V_2, V_3$ be three pairwise disjoint alphabets. Then any word of the form $W=W_1W_2\dots W_s$, where $W_i$ is any word over alphabet $V_{a_i}$ consisting of pairwise distinct letters, is square-free.
\end{lemma}
\begin{proof}Indeed, the word $W$ can be seen as an image of $T$ in a multi-substitution, where for each letter $i\in \mathbb A$ we may put any word over $V_i$ with pairwise distinct letters. Such substitution obviously preserve square-freenes since alphabets $V_i$ are pairwise disjoint.
	\end{proof}
\begin{lemma}\label{Lemma Walks}
	Let $D$ be digraph whose vertices can be partitioned into three sets $V_1,V_2$, and $V_3$ so that the following property holds:
	\item[(*)] For every pair $i,j\in \{1,2,3\}$ and any vertex $v\in V_i$, there is a directed path $P=u_1u_2\dots u_t$ such that $u_1=v$, $\{u_2,\dots, u_{t-1}\}\subseteq V_i$, and $u_t\in V_j$.
	
	Then there exists arbitrarily long square-free walks in $D$.
\end{lemma}
\begin{proof}
	Let us take any square-free word $T=a_1a_2\dots a_s$ over alphabet $\mathbb A= \{1,2,3\}$. Let $v=v_1$ be any vertex in $V_{a_1}$. Let $P_1$ be a directed path satisfying condition (*), starting at $v_1$ and ending in some vertex $v_2\in V_{a_2}$. Now take a similar path $P_2$ starting from $v_2$ and ending in some vertex $v_3\in V_{a_3}$. And so on, until we arrive to some vertex $v_s\in V_{a_s}$. In this way, we obtain a walk $$W=P_1'P_2'\dots P'_s,$$where $P_i'=P_i-\{v_{i+1}\}$ and $P_s'=v_s$. By Lemma \ref{Lemma Substitution}, the walk $W$ is square-free.
	\end{proof}
\begin{lemma}\label{Set S(N)}
The set $S(D_N)$ is infinite.
\end{lemma}
\begin{proof} It is not hard to check that the following partition of $V(D_N)$ satisfies condition (*) of Lemma \ref{Lemma Walks} (see Figure \ref{Figure Digraph2}): $$V_1=\{N,N_{\mathtt{bc}},\tilde N,\tilde{N}_{\mathtt{bc}}\}, V_2=\{N_{\mathtt{ab}},N_{\mathtt{abc}},\tilde{N}_{\mathtt{ab}},\tilde{N}_{\mathtt{abc}}\},V_3=\{N_{\mathtt{ac}},N_{\mathtt{acb}},\tilde{N}_{\mathtt{ac}},\tilde{N}_{\mathtt{acb}}\}.$$
	\end{proof}
\begin{figure}[h]
	\center
	\includegraphics[width=1.0\textwidth]{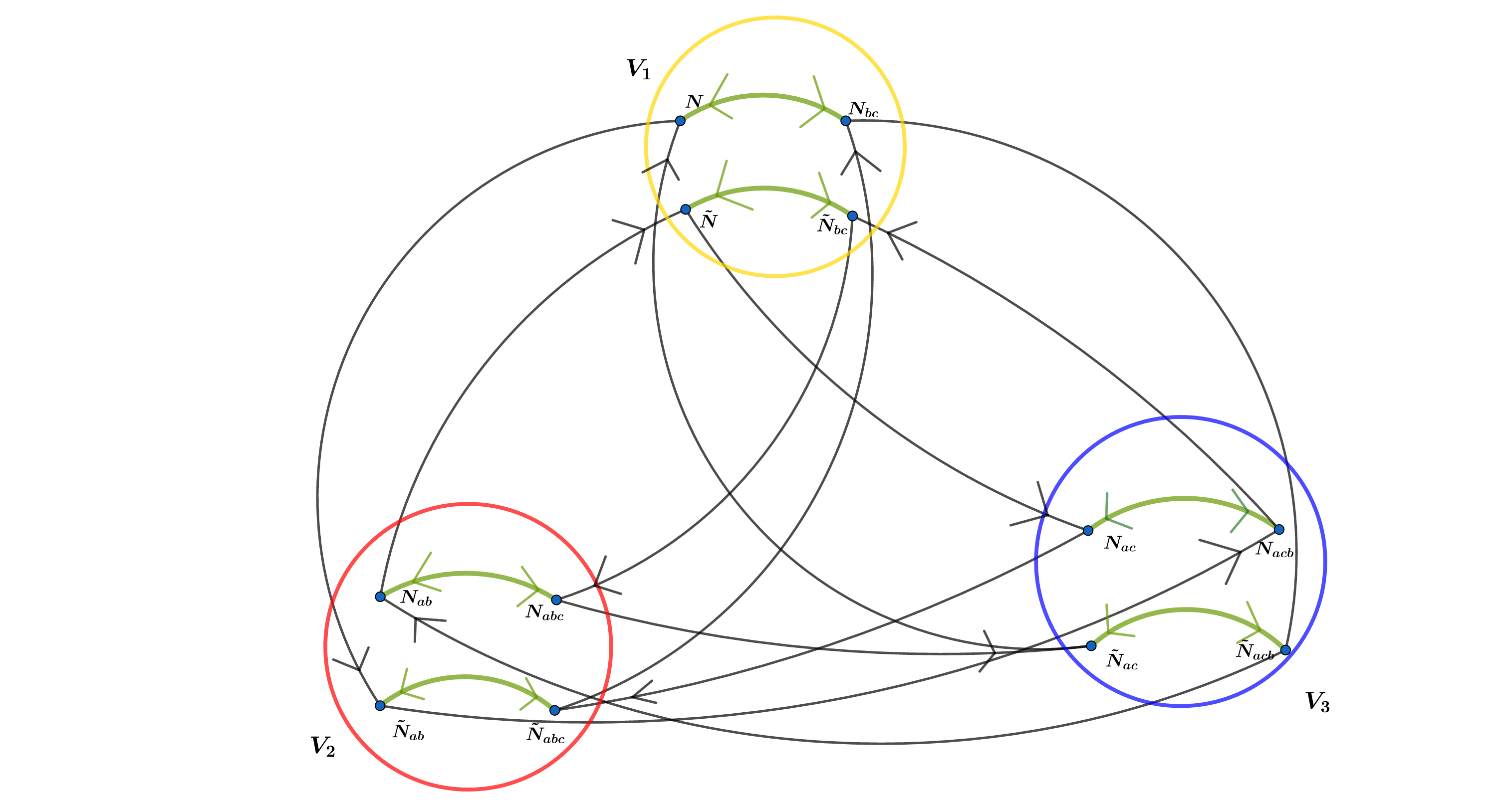}
	\caption{Digraph $D_N$ with vertex partition into three sets.}\label{Figure Digraph2}
\end{figure}
The above four lemmas gives the following conclusion.

\begin{corollary}
There exist arbitrarily long nearly extremal words over $3$-letter alphabet.
\end{corollary}
\begin{proof}It is enough to notice that every word from the set $S(D_N)$ is nearly extremal. Indeed, every such word is square-free and is a concatenation of blocks that are nearly extremal words. Thus it cannot be extended in any inner position of a block, as well as in between any two consecutive blocks. 
	\end{proof}
To prove Theorem \ref{Theorem Extremal Words} we need to modify slightly our digraph $D_N$. The idea is to use two special words: $$P=\mathtt{cbcabcbacbcabacbabcabacbcabc}$$ and $$S=\mathtt {acabcacbabcabacbcabcbacabacbcabcb}.$$The following fact can be checked by computer.

\begin{claim}
The word $PNS$ is extremal.
\end{claim}

Now we may add two new vertices to our digraph $D_N$ with labels $P$ and $S$, and join $P$ to $N$ and $N$ to $S$ by directed edges. Denote this modified digraph as $D_N^*$. The following lemma can also be verified by computer.

\begin{lemma}
The digraph $D_N^*$ is a Thue digraph.
\end{lemma}

To construct long extremal words it is enough to take a sufficiently long square-free walk in $D_N^*$ starting at $P$ and ending in $S$. This completes the proof of Theorem \ref{Theorem Extremal Words}.

\section{Discussion}

We state some further research problems inspired by our result. For instance, one naturally wonders if an analogue of Theorem \ref{Theorem Extremal Words} may be true for larger alphabets. Notice that if we have four letters at a disposal, then there are two potential possibilities for extension of a square-free word in each inner position. Actually, our computer experiment failed in finding extremal words over four letters of length up to 1000. Perhaps there are no such words at all. On the other, hand it is known \cite{BeanEM} that every square-free word (over any alphabet) is a prefix of a \emph{maximal} square-free word, that is, a word non-extendable by attaching a symbol at the beginning or at the end.

\begin{conjecture}
There are no extremal square-free words over $4$-letter alphabet.
\end{conjecture} 

The conjecture states, in other words, that every square-free word over four letters can be extended to a new square-free word by inserting a single letter on some position. This would imply an analogue of Conjecture \ref{Conjecture Nonchalant} for nonchalant words over $4$-letter alphabet.

\begin{conjecture}
The sequence of nonchalant words over $4$-letter alphabet is infinite.
\end{conjecture}

Let us finally mention that similar problems can be considered for other avoidance properties of words. In principle one may investigate extremal words with respect to any given monotone property. Particularly interesting could be \emph{abelian} variants of extremal square-free words.

\end{document}